\documentclass[11pt,a4paper]{amsart}
\usepackage{geometry}
\usepackage[colorlinks, linkcolor=red, anchorcolor=blue, citecolor=blue]{hyperref}
\usepackage{color}             
\usepackage{amsthm,bm} %bm make \mathbf work%
\usepackage{amsfonts}
\usepackage{amsmath}
\usepackage{mathrsfs}
\usepackage{amssymb}
\usepackage{amsbsy}
\usepackage{ifthen}
\usepackage[all]{xy}
\usepackage{extarrows}
\usepackage{indentfirst}        %éŠè¡çŒ©è¿å®å
\usepackage{latexsym}        % å€çæ°å­Šå

\usepackage{graphicx}
\usepackage{cases}
\usepackage{pifont}
\usepackage{txfonts}
\usepackage{xcolor}
\usepackage{multirow}
\usepackage{caption, subcaption}
\usepackage{tikz}

\newcounter{maint}

\numberwithin{equation}{section}

\newcommand{\ydh}{{}^H_H\mathcal{YD}}

\allowdisplaybreaks[4] %% allow to display a terrible long formulae which can't be put in one paper. 

\begin{document}

\newtheorem{theorem}{Theorem}[section]

\newtheorem{lemma}[theorem]{Lemma}

\newtheorem{corollary}[theorem]{Corollary}
\newtheorem{proposition}[theorem]{Proposition}

\theoremstyle{remark}
\newtheorem{remark}[theorem]{Remark}

\theoremstyle{definition}
\newtheorem{definition}[theorem]{Definition}

\theoremstyle{definition}
\newtheorem{conjecture}[theorem]{Conjecture}

\newtheorem{example}[theorem]{Example}

%%%%%%%%%%%%%%%%%%%%%%%%%%%%%%%%%%%%%%%%%%%%%%%
%%%%%%%%%%%%%%%%%%%%%%%%%%%%%%%%%%%%%%%%%%%%%%%%%
\def\k{\Bbbk}
\def\I{\mathbb{I}}
\def\ufo{\mathfrak{ufo}}
\newcommand{\Dchaintwo}[4]{
\rule[-3\unitlength]{0pt}{8\unitlength}
\begin{picture}(14,5)(0,3)
\put(2,4){\ifthenelse{\equal{#1}{l}}{\circle*{4}}{\circle{4}}}
\put(4,4){\line(2,0){20}}
\put(26,4){\ifthenelse{\equal{#1}{r}}{\circle*{4}}{\circle{4}}}
\put(2,10){\makebox[0pt]{\scriptsize #2}}
\put(14,8){\makebox[0pt]{\scriptsize #3}}
\put(26,10){\makebox[0pt]{\scriptsize #4}}
\end{picture}}
\title[The Nichols algebra $\mathfrak{B}(V_{abe})$ and a class of combinatorial numbers]
{The Nichols algebra $\mathfrak{B}(V_{abe})$ and a class of combinatorial numbers}

\author[Shi]{Yuxing Shi }
\address{School of Mathematics and Statistics, Jiangxi Normal University,  Nanchang 330022, P. R. China}\email{yxshi@jxnu.edu.cn}

%\date{October 22, 2013}
\thanks{
%2010 Mathematics Subject Classification: 16T05.\\
\\
This work was partially supported by
 Foundation of Jiangxi Educational Committee (No.12020447)
 and Natural Science Foundation of Jiangxi Normal University(No. 12018937)
}

\begin{abstract}
We investigate the Nichols algebra $\mathfrak{B}(V_{abe})$
which are from the  Yetter-Drinfeld category of  Suzuki algebras.
The  $4n$ and 
$n^2$ dimensional Nichols algebras, first appeared in \cite{Andruskiewitsch2018},  are obtained again via a different method.  
And the connection between the Nichols algebra $\mathfrak{B}(V_{abe})$ and a class of combinatorial numbers on the subgroups of  symmetric groups is established. 
\end{abstract}
\maketitle

\section{introduction}

Nichols algebras appeared for the first time in W. Nichols' paper\cite{nichols1978bialgebras}, which aimed to construct new examples of Hopf algebras.   They also arose independently in 
Woronowicz, Lusztig, and Rosso's works \cite{woronowicz1989differential}\cite{MR1227098}\cite{MR1632802}, via the contexts of non-commutative differential calculus, 
invariant bilinear form and quantum  shuffles respectively. 

Nichols algebras over group algebras are of group type. Determining all finite dimensional  
group type Nichols algebras is a crucial important step for the classification of pointed Hopf algebras\cite{andruskiewitsch2001pointed}. Heckenberger classified all finite dimensional 
Nichols algebras of diagonal type(Abelian group type)\cite{heckenberger2009classification}, 
via the Weyl groupoid introduced in \cite{MR2207786} and arithmetic root systems. Angiono determined the defining relations of the finite dimensional Nichols algebras of diagonal type in \cite{MR3420518}\cite{Angiono2013}.  In general, it's difficult to decide whether a given Nichols algebra of non-diagonal  type is finite dimensional. Andruskiewitsch and Gra{\~n}a established a bridge between group type Nichols algebras and racks
\cite{Andruskiewitsch2003MR1994219}, so the study on  Nichols algebras of non-Abelian group type(also called rack type) can be carried out in the framework of racks. 

When we dealt with Nichols algebras over the Suzuki algebra $A_{Nn}^{\mu\lambda}$\cite{Shi2019}\cite{Shi2020even} 
\cite{Shi2020odd}, we came across the Nichols algebra $\mathfrak{B}(V_{abe})$
of non-group type. After we almost had finished the work, we found that our work is overlapped with 
Andruskiewitsch
 and Giraldi's work \cite[section 3.7]{Andruskiewitsch2018}. 
 Since our approach is different and the problem of determining the dimension of 
 $\mathfrak{B}(V_{abe})$ is not completely solved, our work is still valuable. 
 Besides, we establish a connection  between the Nichols algebra $\mathfrak{B}(V_{abe})$ and a class of combinatorial numbers on the subgroups of the symmetric groups. 

\section{the Nichols algebra $\mathfrak{B}(V_{abe})$}
\newcommand{\F}{\mathfrak{F}}
\newcommand{\tildeF}{\tilde{\mathfrak{F}}}
\newcommand{\G}{\mathcal{G}}

\subsection{Nichols algebra}
\begin{definition}\cite[Definition 2.1]{andruskiewitsch2001pointed}%\cite[Def. 2.1]{AS} 
\label{defNicholsalgebra}
Let $H$ be a Hopf algebra and $V \in \ydh$. A braided $\mathbb{N}$-graded 
Hopf algebra $R = \bigoplus_{n\geq 0} R(n) \in \ydh$  is called 
the \textit{Nichols algebra} of $V$ if 
\begin{enumerate}
 \item[(i)] $\k \simeq R(0)$, $V\simeq R(1) \in \ydh$,
 \item[(ii)] $R(1) = \mathcal{P}(R)
=\{r\in R~|~\Delta_{R}(r)=r\otimes 1 + 1\otimes r\}$.
 \item[(iii)] $R$ is generated as an algebra by $R(1)$.
\end{enumerate}
In this case, $R$ is denoted by $\mathfrak{B}(V) = \bigoplus_{n\geq 0} \mathfrak{B}^{n}(V) $.    
\end{definition}
\begin{remark}
The Nichols algebra 
$\mathfrak{B}(V)$ is completely determined by the braiding.
More precisely, as proved in  \cite{MR1396857} and
noted in \cite{andruskiewitsch2001pointed},
$$\mathfrak{B}(V)=\mathrm{K}\oplus V\oplus\bigoplus\limits_{n=2}^\infty V^{\otimes n} /
\ker\mathfrak{S}_n=T(V)/\ker\mathfrak{S},$$
where $\mathfrak{S}_{n,1}\in \mathrm{End}_\k \left(V^{\otimes (n+1)}\right)$, 
$\mathfrak{S}_{n}\in \mathrm{End}_\k \left(V^{\otimes n}\right)$,
$$\mathfrak{S}_{n,1}\coloneqq\mathrm{id}+c_n+c_{n-1}c_n+\cdots
+c_1\cdots c_{n-1}c_n=\mathrm{id}+\mathfrak{S}_{n-1,1}c_n,$$
$$\mathfrak{S}_1\coloneqq\mathrm{id}, \quad \mathfrak{S}_2\coloneqq\mathrm{id}+c, \quad
\mathfrak{S}_n\coloneqq \mathfrak{S}_{n-1,1}(\mathfrak{S}_{n-1}\otimes \mathrm{id}).$$
\end{remark}

\subsection{The Nichols algebra $\mathfrak{B}(V_{abe})$}
\begin{lemma}\label{undiagonal_BVS}
Let $V=\k v_1\oplus \k v_2$ be a vector space,  
$c\in \mathrm{End}_{\k }\left(V\otimes V\right)$ and suppose  $ab\gamma e\neq 0$ such that  
\begin{align*}
c(v_1\otimes v_1)&=a v_2\otimes v_2,\quad 
&c(v_1\otimes v_2)&=b   v_1\otimes v_2,\\
c(v_2\otimes v_1)&=\gamma v_2\otimes v_1,\quad 
&c(v_2\otimes v_2)&=e   v_1\otimes v_1.
\end{align*}
\begin{enumerate}
\item $(V,c)$ is a braided vector space iff  $b=\gamma$. 
\item The braided vector space $(V,c)$ is of diagonal type iff $b^2=ae$. 
\item Suppose $b=\gamma$,  $b^2=ae$, denote 
         \[
        w_1=v_1+\sqrt{\frac{a}{b}}v_2,\quad  w_2=v_1-\sqrt{\frac{a}{b}}v_2, 
         \] 
       then $(V,c)$ is a braided vector of diagonal type, where the braiding $c$ on the base $(w_1, w_2)$
       is given by 
       \begin{align*}
       c(w_1\otimes w_1)&=b w_1\otimes w_1,\quad 
       &c(w_1\otimes w_2)&=-b   w_2\otimes w_1,\\
       c(w_2\otimes w_1)&=-b w_1\otimes w_2,\quad 
       &c(w_2\otimes w_2)&=b   w_2\otimes w_2.
      \end{align*}
\end{enumerate}
\end{lemma}
\begin{remark}
We denote the braided vector space as $V_{abe}$ in following. And $V_{abe}$ is isomorphic to 
$V_{ae\,b\,1}$ via $v_1\mapsto \sqrt{e}v_1$, $v_2\mapsto v_2$.
\end{remark}

\begin{corollary}
Let  $b^2=ae$, then  $\mathfrak{B}(V_{abe})$ is of diagonal type and 
\[
\dim \mathfrak{B}(V_{abe})
=\left\{\begin{array}{ll}
4, & b=-1, (\mathfrak{B}(V_{abe}) \,\,\,\text{is of Cartan type $A_1\times A_1$}),\\
27, & b^3=1\neq b, (\mathfrak{B}(V_{abe})\,\,\, \text{is of Cartan type $A_2$}),\\
\infty, & \text{otherwise}.
\end{array}\right.
\]
\end{corollary}

\subsection{The dimension of  $\mathfrak{B}(V_{abe})$ on special cases}
Let $N=\{1, 2\}$,   define an  action of the symmetric group $\mathbb{S}_2$ on $N^2$ as 
\[
s_1\cdot 11=22,\quad s_1\cdot 22=11,\quad
s_1\cdot 12=12,\quad s_1\cdot 21=21.
\]
Then there is an action of the symmetric group $\mathbb{S}_n$ on $N^n$ induced by the action of 
$\mathbb{S}_2$ on $N^2$. $\Phi: \mathbb{S}_n\rightarrow \mathfrak{S}_n$ via $s_i\mapsto c_i$
is a lifting of $\mathbb{S}_n$ on the braid group $\mathfrak{S}_n$.    
As for convenience, we denote $\Phi_\sigma$  as $\Phi(\sigma)$  for any 
$\sigma \in \mathbb{S}_n$. 

Let $x=i_1\cdots i_n$, $y=j_1\cdots j_n\in N^n$, define 
$v_x=v_{i_1}\cdots v_{i_n}$, $v_y=v_{j_1}\cdots v_{j_n}\in V_{abe}^n$,
$\F(x|y)=\{\sigma\in \mathbb{S}_n\mid \sigma\cdot x=y\}$ and 
$\tildeF(x|y)\in \k$ such that 
\[
\sum_{\sigma\in \F(x|y)}\Phi_\sigma\left(v_x\right)=\tildeF(x|y)v_y.
\]
Let $\bar{i}_k$ equal to 2 in case of $i_k=1$ and 1 otherwise  for $1\leq k\leq n$, and 
$\bar{x}=\bar{i}_1\cdots \bar{i}_n$. Denote 
$
\mathcal{O}(x)=\{y\mid \sigma\cdot x=y, \sigma\in \mathbb{S}_n\},
$
Then it's easy to see that 
\begin{lemma}\label{tildefxy}
\begin{enumerate}
\setlength{\itemindent}{-0.6em}
\item $\mathfrak{S}_n(v_x)=\sum_{y\in \mathcal{O}(x)}\tildeF(x|y)v_y$,
\item $\tildeF(x|y)=\tildeF(y|x)\big|_{a=e \atop e=a}$,  
\item $\tildeF(\bar{x}|\bar{y})
          =\tildeF(x|y)\big|_{a=e \atop e=a}$.
\item Denote $\vec{x}=i_n\cdots i_1$ if $x=i_1\cdots i_n$, 
then $\tildeF(x|y)=\tildeF(\vec{x}\, | \vec{y})$.
\end{enumerate}
\end{lemma}
\begin{proof}
(1) It's easy to see. \\
(2) $\F(y|x)=\left\{w^{-1}\mid w\in \F(x|y) \right\}$.\\
(3) $\F(\bar{x}|\bar{y})=\F(x|y)$.\\
(4) There is a bijection between $\F(x|y)$ and $\F(\vec{x}\, | \vec{y})$
via a flip of the diagrams of braidings.  
\end{proof}

\begin{lemma}\label{shuffle}
\begin{align*}
\F\left(2^m(12)^k|2^m(12)^k\right)
&=\left(id^{\otimes m}\otimes \mathbb{S}_{2k-1, 1}\right)
     \F\left(2^{m-1}(21)^{k}|2^{m-1}(21)^k\right)\\
&\quad \cup\bigcup_{k_1=0}^{k_1\leq \frac{m-1}2}\prod_{i=m-2k_1}^{m+2k-1}s_i
             \F\left(2^{m-1}(21)^{k}| 2^{m-2k_1-1}1^{2k_1}(21)^{k}\right),\\
\F\left(2^m(21)^k|2^m(21)^k\right)
&=\left(id^{\otimes m}\otimes \mathbb{S}_{2k-1, 1}\right)
     \F\left(2^{m+1}(12)^{k-1}|2^{m+1}(12)^{k-1}\right)\\
&\quad \cup\bigcup_{k_1=0}^{k_1\leq \frac m2}\prod_{i=m-2k_1+1}^{m+2k-1}s_i
             \F\left(2^{m+1}(12)^{k-1}| 2^{m-2k_1}1^{2k_1}2(12)^{k-1}\right).     
\end{align*}
\end{lemma}
\begin{proof}
We give a proof for the first formula. Since $\mathbb{S}_{m+2k}=\mathbb{S}_{m+2k-1, 1}\left(\mathbb{S}_{m+2k-1}\otimes 1\right)$, we only need to make clear which elements of $N^{m+2k}$ are sent to $2^m(12)^k$ under the action of $\mathbb{S}_{m+2k-1, 1}$. In fact,  $2^m(12)^k$ is stable under the action of $id^{\otimes m}\otimes \mathbb{S}_{2k-1, 1}$, and the rest elements are shown in Figure \ref{2^m(12)^k}. The second formula is similar to prove, please refer the Figure \ref{2^m(21)^k}.
\end{proof}
%%%%%%%%%%%%%%%%%%%%%%%%%%
%%%%%%%%%%%%%%%%%%%%%%%%%%
%%%%%%%%%%%%%%%%%%%%%%%%%%
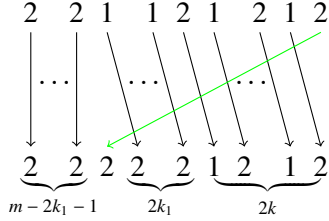
\begin{figure}[h!]
\begin{tikzpicture}
\draw[->](0,1.5)--(0,0);
\draw[->](0.6,1.5)--(0.6,0);
\draw[->](1,1.5)--(1.4,0);
\draw[->](1.6, 1.5)--(2, 0);
\draw[->](2, 1.5)--(2.4, 0);
\draw[->](2.4, 1.5)--(2.8, 0);
\draw[->](3, 1.5)--(3.4, 0);
\draw[->](3.4, 1.5)--(3.8, 0);
\draw[color=green][->](3.8, 1.5)--(1, 0);
\node[above] at (0,1.5) {$2$};
\node[above] at (0.6,1.5) {$2$}; 
\node[above] at (1, 1.5) {$1$};
\node[above] at (1.6, 1.5) {$1$};
\node[above] at (2, 1.5) {$2$};
\node[above] at (2.4, 1.5) {$1$};
\node[above] at (3, 1.5) {$2$};
\node[above] at (3.4, 1.5) {$1$};
\node[above] at (3.8, 1.5) {$2$};
\node[below] at (0,0) {$2$}; 
\node[below] at (0.6,0) {$2$};
\node[below] at (1,0) {$2$};
\node[below] at (1.4, 0) {$2$};
\node[below] at (2, 0) {$2$};
\node[below] at (2.4, 0) {$1$};
\node[below] at (2.8, 0) {$2$};
\node[below] at (3.4, 0) {$1$};
\node[below] at (3.8, 0) {$2$};
\node at (0.34, 0.8) {$\cdots$};
\node at (1.51, 0.8) {$\cdots$};
\node at (2.92, 0.8) {$\cdots$};
\node[rotate = 0] at (0.3, -0.5) {$\underbrace{\hspace{0.4cm}}$};
\node at (0.3, -0.8) {\tiny $m-2k_1-1$};
\node[rotate = 0] at (1.7, -0.5) {$\underbrace{\hspace{0.4cm}}$};
\node at (1.7, -0.8) {\tiny $2k_1$};
\node[rotate = 0] at (3.1, -0.5) {$\underbrace{\hspace{1.4cm}}$};
\node at (3.1, -0.8) {\tiny $2k$};
\end{tikzpicture}
\caption{$\F\left(2^{m}(12)^k|2^{m}(12)^k\right)$}
\label{2^m(12)^k}
\end{figure}
%%%%%%%%%%%%%%
%%%%%%%%%%%%%%%%%%%%%%%%%%
%%%%%%%%%%%%%%%%%%%%%%%%%%
%%%%%%%%%%%%%%%%%%%%%%%%%%
%%%%%%%%%%%%%%%%%%%%%%%%%%
%%%%%%%%%%%%%%%%%%%%%%%%%%
%%%%%%%%%%%%%%%%%%%%%%%%%%
%%%%%%%%%%%%%%%%%%%%%%%%%%
%%%%%%%%%%%%%%%%%%%%%%%%%%
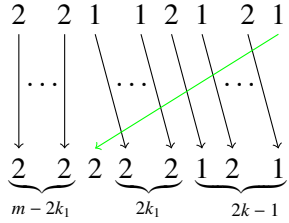
\begin{figure}[h!]
\begin{tikzpicture}
\draw[->](0,1.5)--(0,0);
\draw[->](0.6,1.5)--(0.6,0);
\draw[->](1,1.5)--(1.4,0);
\draw[->](1.6, 1.5)--(2, 0);
\draw[->](2, 1.5)--(2.4, 0);
\draw[->](2.4, 1.5)--(2.8, 0);
\draw[->](3, 1.5)--(3.4, 0);
%\draw[->](3.4, 1.5)--(3.8, 0);
\draw[color=green][->](3.4, 1.5)--(1, 0);
\node[above] at (0,1.5) {$2$};
\node[above] at (0.6,1.5) {$2$}; 
\node[above] at (1, 1.5) {$1$};
\node[above] at (1.6, 1.5) {$1$};
\node[above] at (2, 1.5) {$2$};
\node[above] at (2.4, 1.5) {$1$};
\node[above] at (3, 1.5) {$2$};
\node[above] at (3.4, 1.5) {$1$};
%\node[above] at (3.8, 1.5) {$2$};
\node[below] at (0,0) {$2$}; 
\node[below] at (0.6,0) {$2$};
\node[below] at (1,0) {$2$};
\node[below] at (1.4, 0) {$2$};
\node[below] at (2, 0) {$2$};
\node[below] at (2.4, 0) {$1$};
\node[below] at (2.8, 0) {$2$};
\node[below] at (3.4, 0) {$1$};
%\node[below] at (3.8, 0) {$2$};
\node at (0.34, 0.8) {$\cdots$};
\node at (1.51, 0.8) {$\cdots$};
\node at (2.92, 0.8) {$\cdots$};
\node[rotate = 0] at (0.3, -0.5) {$\underbrace{\hspace{0.4cm}}$};
\node at (0.3, -0.8) {\tiny $m-2k_1$};
\node[rotate = 0] at (1.7, -0.5) {$\underbrace{\hspace{0.4cm}}$};
\node at (1.7, -0.8) {\tiny $2k_1$};
\node[rotate = 0] at (2.91, -0.5) {$\underbrace{\hspace{1.2cm}}$};
\node at (3.1, -0.8) {\tiny $2k-1$};
\end{tikzpicture}
\caption{$\F\left(2^{m}(21)^k|2^{m}(21)^k\right)$}
\label{2^m(21)^k}
\end{figure}
%%%%%%%%%%%%%%
%%%%%%%%%%%%%%%%%%%%%%%%%%
%%%%%%%%%%%%%%%%%%%%%%%%%%
%%%%%%%%%%%%%%%%%%%%%%%%%%

\begin{lemma}\label{GroupF}
$\F\left(2^n|2^n\right)$ is subgroup of $\mathbb{S}_n$ generated by $t_i=s_is_{i+1}s_i$ for 
$1\leq i\leq n-2$. 
\end{lemma}
\begin{proof}
When $n=2$, $\F\left(2^n|2^n\right)=\{1\}=\mathbb{S}_1$. When $n=3$, $\F\left(2^n|2^n\right)=\{1, s_1s_2s_1\}$. According to Lemma \ref{shuffle}, 
\begin{align*}
\F\left(2^n|2^n\right)
&=\bigcup_{k_1=0}^{k_1\leq \frac{n-1}2}\prod_{i=n-2k_1}^{n-1}s_i
             \F\left(2^{n-1}| 2^{n-2k_1-1}1^{2k_1}\right)\\
&=\bigcup_{k_1=0}^{k_1\leq \frac{n-1}2}\prod_{i=n-2k_1}^{n-1}s_i
      \prod_{k_2=0}^{k_1-1}s_{n-2(k_1-k_2)}\F\left(2^{n-1}| 2^{n-1}\right)\\
&=\bigcup_{k_1=0}^{k_1\leq \frac{n-1}2}        
      \prod_{k_2=0}^{k_1-1}t_{n-2(k_1-k_2)}\F\left(2^{n-1}| 2^{n-1}\right).    
\end{align*}
The Lemma is proved by induction.
\end{proof}
\begin{corollary}
$\big|\mathcal{O}\left(2^{2n}\right)\big|={2n\choose n}$, 
$\big|\mathcal{O}\left(2^{2n+1}\right)\big|={2n+1\choose n}$.
\end{corollary}
\begin{proof}
Let $ 1\leq i\leq n-2$, $T_n=<t_i\mid i\equiv 1\mod 2>$ and 
$T_n^\prime=<t_i\mid i\equiv 0\mod 2>$, then 
$\F\left(2^{2n}|2^{2n}\right)=T_nT^\prime_n=T_n^\prime T_n\simeq \mathbb{S}_n\times \mathbb{S}_n$. So $\big|\F\left(2^{2n}|2^{2n}\right)\big|=(n!)^2$, which implies $\big|\mathcal{O}\left(2^{2n}\right)\big|={2n\choose n}$. Similarly, $\F\left(2^{2n+1}|2^{2n+1}\right)\simeq \mathbb{S}_{n+1}\times \mathbb{S}_n$, so $\big|\mathcal{O}\left(2^{2n+1}\right)\big|={2n+1\choose n}$.
\end{proof}

%%%%%%%%%%%%%%%%%%%%%%%%%%
%%%%%%%%%%%%%%%%%%%%%%%%%%
%%%%%%%%%%%%%%%%%%%%%%%%%%

\begin{proposition}\label{pascalTriangle}
\begin{align}
\big|\mathcal{O}\left(2^{2(n-k)}(12)^k\right)\big|
&={2n\choose n-k},\\
\big|\mathcal{O}\left(2^{2(n-k)+1}(12)^k\right)\big|
&={2n+1\choose n-k}.
\end{align}
\end{proposition}
\begin{proof}
Firstly, we observe that  $\F\left(2^{m-1}(21)^k|2^{m-1}(21)^k\right)$ is a subgroup of $\mathbb{S}_{m+2k}$ and 
$\F\left(2^{m-1}(21)^k|2^{m-1-2k_1}1^{2k_1}(21)^k\right)$ is a coset of $\F\left(2^{m-1}(21)^k|2^{m-1}(21)^k\right)$. According to Lemma \ref{shuffle}, we have 
\[
\big|\F\left(2^{2m}(12)^k|2^{2m}(12)^k\right)\big|
=(m+2k)\big|\F\left(2^{2m-1}(21)^k|2^{2m-1}(21)^k\right)\big|.
\]
Similarly, we can prove that 
\begin{align*}
\big|\F\left(2^{2m-1}(12)^k|2^{2m-1}(12)^k\right)\big|
&=(m+2k)\big|\F\left(2^{2m-2}(21)^k|2^{2m-2}(21)^k\right)\big|, \\
\big|\F\left(2^{2m}(21)^k|2^{2m}(21)^k\right)\big|
&=(m+2k)\big|\F\left(2^{2m+1}(12)^{k-1}|2^{2m+1}(12)^{k-1}\right)\big|, \\
\big|\F\left(2^{2m+1}(21)^k|2^{2m+1}(21)^k\right)\big|
&=(m+2k)\big|\F\left(2^{2m+2}(12)^{k-1}|2^{2m+2}(12)^{k-1}\right)\big|.
\end{align*}
Now we can give a direct computation as follows 
\begin{align*}
&\quad\big|\F\left(2^{2(n-k)}(12)^k|2^{2(n-k)}(12)^k\right)\big|\\
&=(n+k)\big|\F\left(2^{2(n-k)-1}(21)^k|2^{2(n-k)-1}(21)^k\right)\big|\\
&=(n+k)(n+k-1)\big|\F\left(2^{2(n-k)}(12)^{k-1}|2^{2(n-k)}(12)^{k-1}\right)\big|\\
&=(n+k)(n+k-1)\cdots (n-k+1)\big|\F\left(2^{n-k}|2^{n-k}\right)\big|\\
&=(n+k)!(n-k)!\\
%%%%%%%%
&\quad\big|\F\left(2^{2(n-k)+1}(12)^k|2^{2(n-k)+1}(12)^k\right)\big|\\
&=(n-k+1+2k)\big|\F\left(2^{2(n-k)}(21)^k|2^{2(n-k)}(21)^k\right)\big|\\
&=(n+k+1)(n-k+2k)\big|\F\left(2^{2(n-k)+1}(12)^{k-1}|2^{2(n-k)+1}(12)^{k-1}\right)\big|\\
&=(n+k+1)(n+k)\cdots (n-k+2)\big|\F\left(2^{2(n-k)+1}|2^{2(n-k)+1}\right)\big|\\
&=(n+k+1)!(n-k)!
\end{align*}
\end{proof}
%%%%%%%%%%%%%%%%%%%%%%%%%%
%%%%%%%%%%%%%%%%%%%%%%%%%%
%%%%%%%%%%%%%%%%%%%%%%%%%%
%%%%%%%%%%%%%%%%%%%%%%%%%%
%%%%%%%%%%%%%%%%%%%%%%%%%%
%%%%%%%%%%%%%%%%%%%%%
%%%%%%%%%%%%%%%%%%%%%

\begin{theorem}\label{theorem_pascal}
\begin{align*}
N^{2n}&=\mathcal{O}\left(2^{2n}\right)\oplus\bigoplus_{k=1}^n
             \left[\mathcal{O}\left(2^{2(n-k)}(21)^{k}\right)
             \oplus \mathcal{O}\left(1^{2(n-k)}(12)^{k}\right)\right], \\
N^{2n+1}&=\bigoplus_{k=0}^{n}
             \left[\mathcal{O}\left(2^{2(n-k)+1}(12)^{k}\right)
             \oplus \mathcal{O}\left(1^{2(n-k)+1}(21)^{k}\right)\right].
\end{align*}
\end{theorem}
\begin{remark}
The dimensions of the irreducible submodules of $N^n$ forms exactly the Pascal's triangle according to  the Proposition \ref{pascalTriangle} and the Theorem \ref{theorem_pascal}. This is not the first time that Pascal's triangle has been categorified  in a module category, please refer to \cite{MR2541502}, 
\cite{MR3069217}, and \cite{Im2020}. 
%\href{https://arxiv.org/abs/1906.07472}{arXiv:1906.07472}.
\end{remark}
%%%%%%%%%%%%%%%%%%%%%
\begin{proof}
According to Proposition \ref{pascalTriangle}, we only need to check that 
$\mathcal{O}\left(2^{2(n-k)}(21)^{k}\right)
             \neq  \mathcal{O}\left(1^{2(n-k)}(12)^{k}\right)$ and 
$\mathcal{O}\left(2^{2(n-k)+1}(12)^{k}\right)
             \neq \mathcal{O}\left(1^{2(n-k)+1}(21)^{k}\right)$. The latter is obvious since 
the action of $\mathbb{S}_{2n+1}$ on $N^{2n+1}$ always changes even $2$'s  and  $1$'s
into each other.  Suppose $\mathcal{O}\left(2^{2(n-k)}(21)^{k}\right)
             =  \mathcal{O}\left(1^{2(n-k)}(12)^{k}\right)$, 
then 
\[
\mathcal{O}\left(22^{2(n-k)}(21)^{k}\right)
             =  \mathcal{O}\left(21^{2(n-k)}(12)^{k}\right).
\] 
This is a contradiction, since 
\begin{align*}
\mathcal{O}\left(21^{2(n-k)}(12)^{k}\right)
          &=\mathcal{O}\left(2^{2(n-k)+1}(12)^{k}\right), \\
\mathcal{O}\left(22^{2(n-k)}(21)^{k}\right)
          &=  \mathcal{O}\left(1^{2(n-k)}2(21)^{k}\right)
             =  \mathcal{O}\left(1^{2(n-k+1)+1}(21)^{k-1}\right).
\end{align*} 
\end{proof}
%%%%%%%%%%%%%%%%%%%%%
%%%%%%%%%%%%%%%%%%%%%%%%%%
%%%%%%%%%%%%%%%%%%%%%%%%%%
%%%%%%%%%%%%%%%%%%%%%%%%%%
Denote $(n)_b=1+b+\cdots +b^{n-1}=\frac{b^n-1}{b-1}$, $(n)_b^!=\prod_{k=1}^n(k)_b$.
\begin{theorem}
Suppose $b^2\neq a=1$, then   
$\dim \mathfrak{B}\left(V_{1b1}\right)< \infty$ if and only if $b$ is a $n$-th primitive 
root of unity for $n\geq 2$.  In particular, $\dim \mathfrak{B}\left(V_{1b1}\right)=n^2$.
\end{theorem}
%%%%%%%%%%%%%%%%%%
\begin{proof}
When $a=e=1$,  
$
\mathfrak{S}_2(v_1^2)=v_1^2+v_2^2=\mathfrak{S}_2(v_2^2)$,  so $\mathfrak{S}_n(v_x)=\mathfrak{S}_n(v_y)$ for any $y\in\mathcal{O}(x)\subset N^n$. That is to say, $\tildeF(x_1|y_1)=\tildeF(x_2|y_2)$  for $x_1, x_2, y_1, y_2\in \mathcal{O}(x_1)\subset N^n$ and $\mathfrak{S}_n(v_x)=\tildeF(x|x)\sum\limits_{y\in \mathcal{O}(x)}v_y$.
%%%%%%%%%%%%%%%%%%%%%%%%%%
\begin{align*}
&\quad\tildeF\left(2^{2(n-k)}(12)^k|2^{2(n-k)}(12)^k\right)\\
&=(n+k)_b\tildeF\left(2^{2(n-k)-1}(21)^k|2^{2(n-k)-1}(21)^k\right)\\
&=(n+k)_b(n+k-1)_b\tildeF\left(2^{2(n-k)}(12)^{k-1}|2^{2(n-k)}(12)^{k-1}\right)\\
&=(n+k)_b(n+k-1)_b\cdots (n-k+1)_b\tildeF\left(2^{n-k}|2^{n-k}\right)\\
&=(n+k)^!_b(n-k)^!_b\\
%%%%%%%%
&\quad\tildeF\left(2^{2(n-k)+1}(12)^k|2^{2(n-k)+1}(12)^k\right)\\
&=(n-k+1+2k)_b\tildeF\left(2^{2(n-k)}(21)^k|2^{2(n-k)}(21)^k\right)\\
&=(n+k+1)_b(n-k+2k)_b\tildeF\left(2^{2(n-k)+1}(12)^{k-1}|2^{2(n-k)+1}(12)^{k-1}\right)\\
&=(n+k+1)_b(n+k)_b\cdots (n-k+2)_b\tildeF\left(2^{2(n-k)+1}|2^{2(n-k)+1}\right)\\
&=(n+k+1)^!_b(n-k)^!_b
\end{align*}
%%%%%%%%%%%%%%%%%%%%%%%%%%
Now we can put a basis of $\mathfrak{B}\left(V_{1b1}\right)$ on the Pascal's triangle and the 
coefficients of the symmetrizer's action on the basis is given by the Figure \ref{coefficients}.
%%%%%%%%%%%%%%%%%%%%%%%%%%
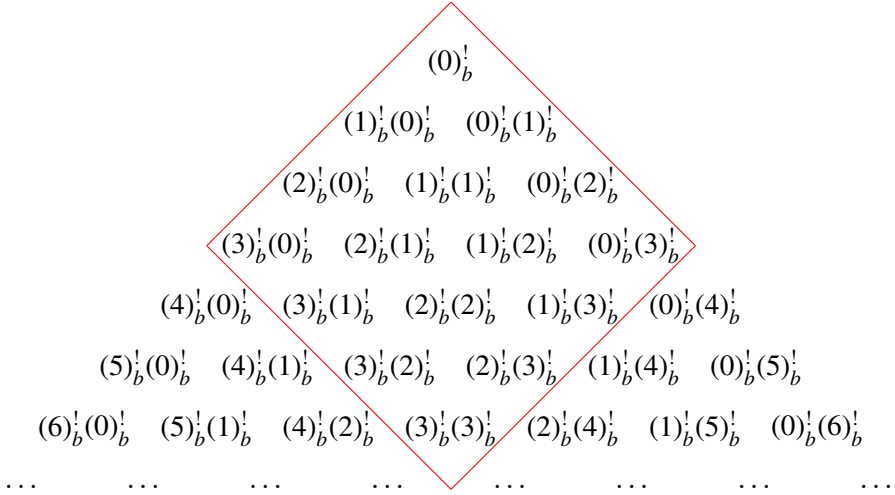
\begin{figure}[h!]
\begin{tikzpicture}
\node at (0, 0) {$(0)_b^!$};
\node at (-0.8, -0.8) {$(1)_b^!(0)_b^!$};
\node at (0.8, -0.8) {$(0)_b^!(1)_b^!$};
\node at (0, -1.6) {$(1)_b^!(1)_b^!$};
\node at (-1.6, -1.6) {$(2)_b^!(0)_b^!$};
\node at (1.6, -1.6) {$(0)_b^!(2)_b^!$};
\node at (-0.8, -2.4) {$(2)_b^!(1)_b^!$};
\node at (0.8, -2.4) {$(1)_b^!(2)_b^!$};
\node at (-2.4, -2.4) {$(3)_b^!(0)_b^!$};
\node at (2.4, -2.4) {$(0)_b^!(3)_b^!$};
\node at (0, -3.2) {$(2)_b^!(2)_b^!$};
\node at (-1.6, -3.2) {$(3)_b^!(1)_b^!$};
\node at (1.6, -3.2) {$(1)_b^!(3)_b^!$};
\node at (-3.2, -3.2) {$(4)_b^!(0)_b^!$};
\node at (3.2, -3.2) {$(0)_b^!(4)_b^!$};
\node at (-0.8, -4) {$(3)_b^!(2)_b^!$};
\node at (0.8, -4) {$(2)_b^!(3)_b^!$};
\node at (-2.4, -4) {$(4)_b^!(1)_b^!$};
\node at (2.4, -4) {$(1)_b^!(4)_b^!$};
\node at (-4, -4) {$(5)_b^!(0)_b^!$};
\node at (4, -4) {$(0)_b^!(5)_b^!$};
\node at (0, -4.8) {$(3)_b^!(3)_b^!$};
\node at (-1.6, -4.8) {$(4)_b^!(2)_b^!$};
\node at (1.6, -4.8) {$(2)_b^!(4)_b^!$};
\node at (-3.2, -4.8) {$(5)_b^!(1)_b^!$};
\node at (3.2, -4.8) {$(1)_b^!(5)_b^!$};
\node at (-4.8, -4.8) {$(6)_b^!(0)_b^!$};
\node at (4.8, -4.8) {$(0)_b^!(6)_b^!$};
\node at (-0.8, -5.6) {$\cdots$};
\node at (0.8, -5.6) {$\cdots$};
\node at (-2.4, -5.6) {$\cdots$};
\node at (2.4, -5.6) {$\cdots$};
\node at (-4, -5.6) {$\cdots$};
\node at (4, -5.6) {$\cdots$};
\node at (-5.6, -5.6) {$\cdots$};
\node at (5.6, -5.6) {$\cdots$};
\draw[color=red][-](0, 0.8)--(3.2, -2.4);
\draw[color=red][-](0, -5.6)--(3.2, -2.4);
\draw[color=red][-](0, -5.6)--(-3.2, -2.4);
\draw[color=red][-](0, 0.8)--(-3.2, -2.4);
\end{tikzpicture}
\caption{coefficients of the braided symmetrizer's action on the basis of $\mathfrak{B}\left(V_{1b1}\right)$}
\label{coefficients}
\end{figure}
When $b$ is a primitive $n$-th root of unity for $n\geq 2$, then all the nonzero coefficients located in a square 
on the top of the Pascal's triangle, see the case when $n=3$ in Figure \ref{coefficients}. So 
$\dim\mathfrak{B}\left(V_{1b1}\right)=n^2$.
\end{proof}
%%%%%%%%%%%%%%%%%%%%%%%%%%
%%%%%%%%%%%%%%%%%%%%%%%%%%
%%%%%%%%%%%%%%%%%%%%%%%%%%
%%%%%%%%%%%%%%%%%%%%%
%%%%%%%%%%%%%%%%%%%%%
%%%%%%%%%%%%%%%%%%%%%
\begin{theorem}
Suppose $b^2\neq ae$ and $b=-1$, then $\mathfrak{B}\left(V_{abe}\right)<\infty$ if
$ae$ is a $m$-th primitive root of unity for $m\geq 1$. In particular, 
$\dim \mathfrak{B}\left(V_{abe}\right)=4m$. 
\end{theorem}
\begin{proof}
Since $b=-1$, $v_1v_2=v_2v_1=0$. Let $x=i_1\cdots i_n\in N^n$ such that 
$i_ki_{k+1}=12$ or $21$ for some $k=0,\cdots, n-1$, then 
\[
0=\mathfrak{S}_n(v_x)=\sum_{y\in\mathcal{O}(x)}\tildeF(x,y)v_y
\Rightarrow \tildeF(x,y)=0. 
\]
According to Lemma \ref{tildefxy}, we have $\tildeF(y, x)=0$ for any $y\in \mathcal{O}(x)$.
So 
\begin{align*}
\mathfrak{S}_n(v_1^n)=\left\{\begin{array}{ll}
\tildeF(1^n, 1^n)v_1^n+\tildeF(1^n, 2^n)v_2^n, & n\,\,\, \text{is even}, \\
\tildeF(1^n, 1^n)v_1^n, & n\,\,\, \text{is odd}. 
\end{array}\right.
\end{align*}
Since $\mathfrak{S}_n= \mathfrak{S}_{n-1,1}(\mathfrak{S}_{n-1}\otimes \mathrm{id})$, 
we have 
\begin{align*}
\tildeF\left(1^{2m},2^{2m}\right)
&=a(ab)^{m-1}\tildeF\left(1^{2m-1},1^{2m-1}\right)\\
%=a(-a)^{m-1}\tildeF\left(1^{2m-1},1^{2m-1}\right)\\
%%%%%%%%%
\tildeF\left(1^{2m},1^{2m}\right)
&=\tildeF\left(1^{2m-1},1^{2m-1}\right)\\
&=\tildeF\left(1^{2m-2},1^{2m-2}\right)
     +(eb)^{m-1}\tildeF\left(1^{2m-2},2^{2m-2}\right)\\
&=\tildeF\left(1^{2m-2},1^{2m-2}\right)
     +(eb)^{m-1}a(ab)^{m-2}\tildeF\left(1^{2m-3},1^{2m-3}\right)\\
%&=\tildeF\left(1^{2m-2},1^{2m-2}\right)
%     +(ae)^{m-1}b^{2m-3}\tildeF\left(1^{2m-3},1^{2m-3}\right)\\
&=\left[1+(ae)^{m-1}b^{2m-3}\right]\tildeF\left(1^{2m-3},1^{2m-3}\right)\\
&=\left[1-(ae)^{m-1}\right]\tildeF\left(1^{2m-3},1^{2m-3}\right)\\
%%%%%%%%%
\tildeF\left(1^{2m+1},1^{2m+1}\right)
&=\left[1-(ae)^m\right]\left[1-(ae)^{m-1}\right]\cdots 
\left[1-(ae)^2\right]\tildeF\left(1^{3},1^{3}\right)\\
&=\left[1-(ae)^m\right]\left[1-(ae)^{m-1}\right]\cdots 
\left[1-(ae)^2\right]\left(1-ae\right)\\
&=(1-ae)^m(m)_{ae}^!
\end{align*}
Similarly, we have 
 \begin{align*}
 \tildeF\left(2^{2m},2^{2m}\right)
&=\tildeF\left(2^{2m-1},2^{2m-1}\right)
=(1-ae)^{m-1}(m-1)_{ae}^!\\
 \tildeF\left(2^{2m},1^{2m}\right)
&=e(eb)^{m-1}\tildeF\left(2^{2m-1},2^{2m-1}\right)
=e^mb^{m-1}(1-ae)^{m-1}(m-1)_{ae}^!
 \end{align*}
 Since  $ae$ is a $m$-th primitive root,  $v_1^{2m+1}=0$, $v_2^{2m+1}=0$, and 
 \begin{align*}
 \mathfrak{S}_{2m}(v_1^{2m})
 &=(1-ae)^{m-1}(m-1)_{ae}^!\left[v_1^{2m}+a^mb^{m-1}v_2^{2m}\right]\\
  \mathfrak{S}_{2m}(v_2^{2m})
 &=(1-ae)^{m-1}(m-1)_{ae}^!\left[v_2^{2m}+e^mb^{m-1}v_1^{2m}\right]\\
 &=e^mb^{m-1} \mathfrak{S}_{2m}(v_1^{2m}).
 \end{align*}
 So $\dim \mathfrak{B}\left(V_{abe}\right)=4m$. 
\end{proof}
%%%%%%%%%%%%%%%%%%%%%
%%%%%%%%%%%%%%%%%%%%%
%%%%%%%%%%%%%%%%%%%%%

\section{The polynomial $\tildeF(1^n,1^n)$  and a class of  combinatorial numbers}
Since $V_{abe}$ is isomorphic to 
$V_{ae\,b\,1}$ as braided vector spaces and $\tildeF(1^n,1^n)$ is a polynomial of  $b$ and $ae$, we set $e=1$ for convenience.  Let $x, y\in N^n$, denote 
\[
\tildeF(x,y)=\sum\limits_{k=0}\tildeF_k(x, y)b^k,
\]
where $\tildeF_k(x, y)$ is a  polynomial of $a$.
Then 
\[
\tildeF(1^n,1^n)=\sum\limits_{k=0}^{n(n-1)/2}\tildeF_k(1^n,1^n)b^k.
\]
%where $\tildeF_k(1^n,1^n)$ is a polynomial of $a$.

Let $w\in \F\left(1^n|1^n\right)$, define $\mathrm{tl}(w)$ as the minimal length of $w$ in expressions of $t_i$'s and $\mathrm{sl}(w)$ as the minimal length of $w$ in expressions of $s_i$'s. Denote
\[
\mathcal{E}_{k,s}^n=\#\left\{\sigma\in\F\left(1^n,1^n\right)\Big| 
\mathrm{tl}(\sigma)=k, \mathrm{sl}(\sigma)=s\right\}.
\]
\begin{proposition}
$\tildeF_k(1^n,1^n)=\sum_{\mathrm{tl}(w)=k}\#\{\mathrm{sl}(w)\}a^{k-[3k-\mathrm{sl}(w)]/2}
=\sum_{i=0}\mathcal{E}_{k,2i+k}^na^i$. 
\end{proposition}
\begin{proof}
According to the Lemma \ref{GroupF}, $\F(1^n,1^n)=<t_i=s_is_{i+1}s_i\mid i=1,\cdots, n-2>$.
In case $n=2m$ is even, then 
$\F(1^{2m},1^{2m})\simeq \mathbb{S}_m\times \mathbb{S}_m$, where the first 
$\mathbb{S}_m$ is the symmetric group on  even numbers and the second 
$\mathbb{S}_m$ is the symmetric group on odd numbers. 
If $a=e=1$, then $\tildeF(1^{2m},1^{2m})=(m)_b^!(m)_b^!$. So for any $w\in \F(1^{2m},1^{2m})$, 
its contribution to $\tildeF(1^{2m},1^{2m})$ is 
$a^{k-[3k-\mathrm{sl}(w)]/2} b^{\mathrm{tl}(w)}$. The similar result holds on case that $n$ is odd. 
\end{proof}

\begin{lemma}
$\tildeF_0(1^n,1^n)=1$, $\tildeF_1(1^n,1^n)=(n-2)a$ for $n\geq 2$.  
\end{lemma}
\begin{corollary}
$
\mathcal{E}_{1,s}^n=\left\{\begin{array}{ll}
n-2, &s=3, n\geq 2,\\
0, &otherwise.
\end{array}\right.
$
\end{corollary}

\begin{lemma}
$\tildeF_{n(n-1)}(1^{2n},1^{2n})=a^{n(n-1)/2}$ and $\tildeF_{n^2}(1^{2n+1},1^{2n+1})=a^{n(n+1)/2}$.
\end{lemma}
\begin{proof}
Let $w$ be the  longest element  of $\F\left(1^{2n}|1^{2n}\right)$ in the sense of $\mathrm{tl}(w)$
, then  
\begin{align*}
w&=(t_{1}t_3\cdots t_{2n-3})(t_1t_3\cdots t_{2n-5})\cdots (t_1t_3)t_1
     (t_{2}t_4\cdots t_{2n-2})(t_2t_4\cdots t_{2n-4})\cdots (t_2t_4)t_2\\
&=(2n-1, 2n-3, \cdots, 3, 1)(2n, 2n-2, \cdots, 4, 2)\\
&=(2n-1, 2n, 2n-3, 2n-2, \cdots, 3, 4, 1, 2).
\end{align*}
So $\mathrm{sl}(w)=\mathrm{inv}(2n-3, 2n-2, 2n-5, 2n-4, \cdots, 3, 4, 1, 2)
=2n(n-1)$, and 
\begin{align*}
\tildeF_{n(n-1)}(1^{2n},1^{2n})=a^{n(n-1)-[3n(n-1)-\mathrm{sl}(w)]/2}
=a^{n(n-1)/2}.
\end{align*}
Let $w^\prime$ be the  longest element  of $\F\left(1^{2n+1}|1^{2n+1}\right)$
 in the sense of $\mathrm{tl}(w^\prime)$, then
 \begin{align*}
 w^\prime&=(t_1t_3\cdots t_{2n-1})w\\
 &=(2n+1, 2n-1, 2n-3, \cdots, 3, 1)(2n, 2n-2, \cdots, 4, 2)\\
&=(2n+1, 2n, 2n-1, 2n-2, \cdots, 4, 3, 2, 1), 
 \end{align*}
 which imply that $\mathrm{tl}(w^\prime)=n^2$ and $\mathrm{sl}(w^\prime)=n(2n+1)$. So
 \begin{align*}
\tildeF_{n^2}(1^{2n+1},1^{2n+1})=a^{n^2-[3n^2-n(2n+1)]/2}
=a^{n(n+1)/2}.
\end{align*}
\end{proof}

\begin{lemma}
%\begin{align*}
$
\tildeF_2(1^n,1^n)
=\left\{\begin{array}{lr}\frac{(n-1)(n-4)}2a^2+(n-3)a, &n\geq 4, \\
0, &n\leq 3.
\end{array}\right.
$
%\end{align*}
\end{lemma}
\begin{proof}
It's easy to see $\tildeF_2(1^4,1^4)=a$. 
When $n\geq 5$, then 
\begin{align*}
\tildeF_2\left(1^n,1^n\right)
&=\tildeF_2\left(1^{n-1},1^{n-1}\right)+\tildeF_1\left(1^{n-1},1^{n-3}2^2\right)+\tildeF_0\left(1^{n-1},1^{n-5}2^4\right)\\
&=\tildeF_2\left(1^{n-1},1^{n-1}\right)+\left[a\tildeF_1\left(1^{n-2},1^{n-2}\right)+
     a\tildeF_0\left(1^{n-2},1^{n-2}\right)\right]+a^2\\
&=\tildeF_2\left(1^{n-1},1^{n-1}\right)+\left(n-3\right)a^2+a\\
&=\tildeF_2\left(1^4,1^4\right)+\left(2+3+\cdots+n-3\right)a^2+\left(n-4\right)a\\
&=\frac{\left(n-1\right)\left(n-4\right)}2a^2+\left(n-3\right)a.
\end{align*}
\end{proof}

\begin{corollary}
$
\mathcal{E}_{2,s}^n=\left\{\begin{array}{ll}
\frac{\left(n-1\right)\left(n-4\right)}2, & s=6, n\geq 4,\\
n-3, &s=4, n\geq 4,\\
0, &otherwise. 
\end{array}\right.
$
\end{corollary}

\begin{lemma}
%\begin{align*}
$
\tildeF_3(1^n,1^n)
=\left\{\begin{array}{ll}
    \frac{(n+1)(n-4)(n-6)}6a^3+\left(n^2-4n-2\right)a^2, & n\geq 6, \\
    3a^2, &n=5,\\
    0, &n\leq 4.
    \end{array}\right.   
$
%\end{align*}
\end{lemma}
\begin{proof}
When $n\geq 7$, then 
\begin{align*}
\tildeF_3\left(1^n,1^n\right)
&=\tildeF_3\left(1^{n-1},1^{n-1}\right)+\tildeF_0\left(1^{n-1},1^{n-7}2^6\right)
     +a\left[\tildeF_2\left(1^{n-2},1^{n-2}\right)\right.\\
&\quad \left.+\tildeF_1\left(1^{n-2},1^{n-2}\right)
    +\tildeF_0\left(1^{n-2},1^{n-6}2^21^2\right)+\tildeF_0\left(1^{n-3},1^{n-7}1^22^2\right)\right]\\
&\quad+a^2\left[\tildeF_0\left(1^{n-2},1^{n-2}\right)+\tildeF_1\left(1^{n-4},1^{n-4}\right)
    +\tildeF_0\left(1^{n-4},1^{n-4}\right)\right]\\
&=\tildeF_3\left(1^{n-1},1^{n-1}\right)+a^3\\
&\quad    + a\left[\frac{(n-3)(n-6)}{2}a^2+(n-5)a+(n-4)a+a+a\right]\\
&\quad    +a^2\left[1+(n-6)a+1\right]\\
&=\tildeF_3\left(1^{n-1},1^{n-1}\right)+\frac{n^2-7n+8}{2}a^3+(2n-5)a^2\\
&=\tildeF_3\left(1^6,1^6\right)+\sum_{k=7}^n\frac{k^2-7k+8}2a^3+\sum_{k=7}^n(2k-5)a^2\\
&=\frac{(n+1)(n-4)(n-6)}6a^3+\left(n^2-4n-2\right)a^2.
\end{align*}
\end{proof}

\begin{corollary}
%\begin{align*}
$
\mathcal{E}_{3,s}^n=\left\{\begin{array}{ll}
\frac{(n+1)(n-4)(n-6)}6, & s=9, n\geq 6,\\
n^2-4n-2, &s=7, n\geq 5,\\
0, &otherwise. 
\end{array}\right.
$
%\end{align*}
\end{corollary}

\begin{lemma}
When $n\leq 8$, then 
\[
\tildeF_4(1^n,1^n)=\left\{\begin{array}{ll}
17a^2+52a^3+2a^4, &n=8,\\
10a^2+19a^3+a^4,&n=7,\\
4a^2+4a^3,&n=6,\\
a^3,&n=5,\\
0,& n\leq 4.
\end{array}\right.
\]
When $n>8$, then $\tildeF_4(1^n,1^n)$ equals
\[
\frac{(n-7)(n^3-7n^2-14n+96)}{24}a^4
  +\frac{n^3-6n^2-13n+80}2a^3+\frac{n^2-n-22}2a^2 .
\]
\end{lemma}
\begin{proof}
We only prove the case $n\geq 9$.
\begin{align*}
\tildeF_4\left(1^n,1^n\right)
&=\tildeF_4\left(1^{n-1},1^{n-1}\right)+\tildeF_0\left(1^{n-1},1^{n-9}2^8\right)
     +a\tildeF_3\left(1^{n-2},1^{n-2}\right)\\
&\quad+a\tildeF_2\left(1^{n-2},1^{n-2}\right)+a\tildeF_0\left(1^{n-2},1^{n-8}2^41^2\right)
     +a\tildeF_0\left(1^{n-3},1^{n-7}2^4\right)\\
&\quad  +a\tildeF_0\left(1^{n-4},1^{n-6}2^2\right)+a^2\tildeF_1\left(1^{n-2},1^{n-2}\right)
     +a^2\tildeF_0\left(1^{n-2},1^{n-2}\right)\\
&\quad +a^2\tildeF_0\left(1^{n-2},1^{n-7}2^21^3\right)+a^2\tildeF_0\left(1^{n-3},1^{n-5}2^2\right)
     +a^2\tildeF_0\left(1^{n-3},1^{n-7}21^22\right)\\
&\quad+a^2\tildeF_2\left(1^{n-4},1^{n-4}\right)+a^2\tildeF_1\left(1^{n-4},1^{n-4}\right)
     +a^2\tildeF_0\left(1^{n-4},1^{n-8}2^21^2\right)\\
&\quad+a^2\tildeF_1\left(1^{n-4},1^{n-4}\right)+\left(a^2+a^3\right)\tildeF_0\left(1^{n-4},1^{n-4}\right)
     +a^2\tildeF_1\left(1^{n-5},1^{n-5}\right)\\
&\quad +a^2\tildeF_0\left(1^{n-5},1^{n-5}\right)+a^2\tildeF_0\left(1^{n-5},1^{n-7}2^2\right)\\
&\quad +a^3\tildeF_1\left(1^{n-6},1^{n-6}\right)+a^3\tildeF_0\left(1^{n-6},1^{n-6}\right)\\
%%%%%%%%%%%%%%%%%%%
&=\tildeF_4(1^{n-1},1^{n-1})+a^4+a\left[\frac{(n-1)(n-6)(n-8)}{6}a^3+(n^2-8n+10)a^2\right]\\
&\quad +a\left[\frac{(n-3)(n-6)}{2}a^2+(n-5)a\right]+a\cdot a^2+a\cdot a^2+a\cdot a\\
&\quad+a^2\cdot (n-4)a+a^2\cdot 1+a^2\cdot a+a^2\cdot a+a^2\cdot a^2\\
&\quad+a^2\left[\frac{(n-5)(n-8)}{2}a^2+(n-7)a\right]+a^2\cdot (n-6)a+a^2\cdot a\\
&\quad+a^2\cdot (n-6)a+\left(a^2+a^3\right)+a^2\cdot (n-7)a\\
&\quad+a^2+a^2\cdot a+a^3\cdot (n-8)a+a^3\\
%%%%%%%%%%%%%%%%%%
&=\tildeF_4\left(1^{n-1},1^{n-1}\right)+\frac{n^3-12n^2+29n+36}{6}a^4
      +\frac{3n^2-15n-6}{2}a^3\\
&\quad +(n-1)a^2\\
&=\tildeF_4\left(1^{8}|1^{8}\right)+\frac{a^4}{6}\sum_{k=9}^n k^3
     +\left(-2a^4+\frac32a^3\right)\sum_{k=9}^n k^2\\
&\quad+\left(\frac{29}{6}a^4-\frac{15}{2}a^3+a^2\right)\sum_{k=9}^n k
     +\left(6a^4-3a^3-a^2\right)(n-9+1)\\
&=\frac{(n-7)(n^3-7n^2-14n+96)}{24}a^4
     +\frac{n^3-6n^2-13n+80}2a^3\\
&\quad+\frac{n^2-n-22}2a^2. 
\end{align*}
\end{proof}

\begin{corollary}
$\mathcal{E}_{4,s}^n=0$, except the following cases. 
\begin{align*}
\mathcal{E}_{4,12}^7&=1,\quad \mathcal{E}_{4,10}^6=4,\quad \mathcal{E}_{4,10}^5=1,\\
\mathcal{E}_{4,s}^n
&=\left\{\begin{array}{ll}
\frac{(n-7)(n^3-7n^2-14n+96)}{24}, &s=12, n\geq 8,\\
\frac{n^3-6n^2-13n+80}2, & s=10, n\geq 7,\\
\frac{n^2-n-22}2, &s=8, n\geq 6.
%0, & otherwise.
\end{array}\right.
\end{align*}
\end{corollary}

\begin{lemma}
When $n<11$, then 
\begin{align*}
\tildeF_5(1^{n},1^{n})
=\left\{\begin{array}{ll}
10a^5+234a^4+226a^3+4a^2,&n=10,\\
4a^5+96a^4+131a^3+3a^2, &n=9,\\
32a^4+62a^3+2a^2, &n=8,\\
10a^4+19a^3+a^2, &n=7,\\
4a^3, &n=6,\\
0, &n<6. 
\end{array}\right.
\end{align*}
When $n\geq 11$, then 
\begin{align*}
\tildeF_5(1^n,1^n)
&=\frac{n^5 - 20n^4 + 75n^3 + 740n^2 - 5716n + 9360}{120}a^5+(n-6)a^2\\
&\quad+\frac{n^4 - 9n^3 - 34n^2 + 474n - 936}{6}a^4
            +\frac{n^3 - n^2 - 62n+172}{2}a^3.
\end{align*}            
\end{lemma}
\begin{proof}
We only prove the case $n\geq 11$. 
%When $n\geq 11$, then 
\begin{align*}
%\tildeF_5(1^n,1^n)&=\tildeF_5(1^{n-1},1^{n-1})+a\tildeF_4(1^{n-2},1^{n-2})
%         +a\tildeF_3(1^{n-2},1^{n-2})\\
%&\quad+a^2\tildeF_3(1^{n-4},1^{n-4})
%             +a^2\tildeF_2(1^{n-2},1^{n-2})
%             +2a^2\tildeF_2(1^{n-4},1^{n-4})\\
%&\quad +a^2\tildeF_2(1^{n-5},1^{n-5})
%             +a^3\tildeF_2(1^{n-6},1^{n-6})\\
\tildeF_5\left(1^n,1^n\right)
&=\tildeF_5\left(1^{n-1},1^{n-1}\right)+a\tildeF_4\left(1^{n-2},1^{n-2}\right)
         +a\tildeF_3\left(1^{n-2},1^{n-2}\right)\\
&\quad+a^2\tildeF_3\left(1^{n-4},1^{n-4}\right)
             +a^2\tildeF_2\left(1^{n-2},1^{n-2}\right)
             +2a^2\tildeF_2\left(1^{n-4},1^{n-4}\right)\\
&\quad +a^2\tildeF_2\left(1^{n-5},1^{n-5}\right)
             +a^3\tildeF_2\left(1^{n-6},1^{n-6}\right)\\
&\quad +(2n-15)a^5+(8n-43)a^4+(4n-5)a^3+a^2\\
&=\tildeF_5\left(1^{n-1},1^{n-1}\right)+a\left\{
    \frac{(n-9)(n^3-13n^2+26n+88)}{24}a^4\right.\\
&\quad \left.+\frac{n^3-12n^2+23n+74}{2}a^3
             +\frac{n^2-5n-16}{2}a^2\right\}\\
&\quad +a\left[\frac{(n-1)(n-6)(n-8)}{6}a^3+(n^2-8n+10)a^2\right]\\
&\quad +a^2\left[\frac{(n-3)(n-8)(n-10)}{6}a^3+(n^2-12n+30)a^2\right]\\
&\quad +\frac{(n-7)(n-10)}{2}a^5+(2n^2-24n+67)a^4+(4n-27)a^3\\
&\quad +(2n-15)a^5+(8n-43)a^4+(4n-5)a^3+a^2\\
&=\tildeF_5\left(1^{n-1},1^{n-1}\right)+\frac{n^4-18n^3+71n^2+234n-1272}{24}a^5\\
&\quad+\frac{4n^3-33n^2-37n+498}{6}a^4
             +\frac{3n^2-5n-60}{2}a^3+a^2\\
&=\tildeF_5\left(1^{10},1^{10}\right)+\sum_{k=11}^n \frac{k^4-18k^3+71k^2+234k-1272}{24}a^5\\
&\quad+\sum_{k=11}^n \left[\frac{4k^3-33k^2-37k+498}{6}a^4  
             +\frac{3k^2-5k-60}{2}a^3+a^2\right]\\    
&=10a^5+234a^4+226a^3+4a^2+ (n-10)a^2\\
&\quad +\frac{(n - 10)(n - 2)(n^3 - 8n^2 - 41n + 408)}{120}a^5 \\
&\quad +\frac{(n - 10)(n^3 + n^2 -24n + 234)}{6}a^4 
        +\frac{(n-10)(n^2+9n+28)}{2}a^3\\
&=\frac{n^5 - 20n^4 + 75n^3 + 740n^2 - 5716n + 9360}{120}a^5+(n-6)a^2\\
&\quad+\frac{n^4 - 9n^3 - 34n^2 + 474n - 936}{6}a^4
            +\frac{n^3 - n^2 - 62n+172}{2}a^3.
\end{align*}
\end{proof}

%\begin{remark}
%\end{remark}
\begin{corollary}
$\mathcal{E}_{5,s}^n=0$, except the following cases. 
\begin{align*}
\mathcal{E}_{5,15}^9&=4,\quad \mathcal{E}_{5,13}^8=32,\quad \mathcal{E}_{5,11}^7=19,
\quad \mathcal{E}_{5,11}^6=4,\\
\mathcal{E}_{5,s}^n
&=\left\{\begin{array}{ll}
\frac{n^5 - 20n^4 + 75n^3 + 740n^2 - 5716n + 9360}{120}, &s=15, n\geq 10,\\
\frac{n^4 - 9n^3 - 34n^2 + 474n - 936}{6}, & s=13, n\geq 9,\\
\frac{n^3 - n^2 - 62n+172}{2}, &s=11, n\geq 8,\\
n-6, &s=9, n\geq 7. 
\end{array}\right.
\end{align*}
\end{corollary}

\end{document}